\documentclass{amsart}

\title{Relative Frobenius algebras are groupoids}
\date{\today}
\author{Chris Heunen}
\address{Department of Computer Science, University of Oxford, Wolfson
  Building, Parks Rd, OX1 3QD, Oxford, United Kingdom}
\thanks{Supported by ONR.}
\email{heunen@cs.ox.ac.uk}
\author{Ivan Contreras}
\address{Institut f{\"u}r Mathematik, Universit{\"a}t Z{\"u}rich,
  Winterthurerstrasse 190, CH-8057, Z{\"u}rich, Switzerland}
\thanks{Partially supported by SNF Grant 20-131813.}
\email{ivan.contreras@math.uzh.ch}
\email{alberto.cattaneo@math.uzh.ch}
\author{Alberto S. Cattaneo}


\usepackage{xypic,amssymb}
\newcommand{\cat}[1]{\ensuremath{\mathbf{#1}}}
\newcommand{\Cat}[1]{\ensuremath{\mathbf{#1}}}
\newcommand{\op}{\ensuremath{^{\mathrm{op}}}}
\newcommand{\after}{\ensuremath{\circ}}
\newcommand{\relto}{\ensuremath{\xymatrix@1@C-2ex{\ar|(.4)@{|}[r] &}}}
\newcommand{\defined}{\ensuremath{\mathop{\downarrow}}}
\newcommand{\rel}{\ensuremath{^{\mathrm{rel}}}}
\newcommand{\func}{\ensuremath{^{\mathrm{func}}}}
\newcommand{\mfunc}{\ensuremath{^{\mathrm{mfunc}}}}
\newcommand{\name}[1]{\ensuremath{\ulcorner #1 \urcorner}}
\newcommand{\swapmap}{\ensuremath{\sigma}}

\theoremstyle{plain}
\newtheorem{theorem}[equation]{Theorem}
\newtheorem{corollary}[equation]{Corollary}
\newtheorem{lemma}[equation]{Lemma}
\newtheorem{proposition}[equation]{Proposition}
\theoremstyle{definition}
\newtheorem{definition}[equation]{Definition}

\usepackage{tikz}
\pgfdeclarelayer{edgelayer}
\pgfdeclarelayer{nodelayer}
\pgfsetlayers{edgelayer,nodelayer,main}
\tikzstyle{none}=[inner sep=1pt]
\tikzstyle{box}=[draw=black, fill=white, inner sep=.5ex, rounded corners=.1ex]
\tikzstyle{cross}=[preaction={draw=white, -, line width=3pt}]
\tikzstyle{dot}=[circle, draw=black, fill=black!50, inner sep=1.25pt]

\begin{document}
\maketitle

\begin{abstract}
  We functorially characterize groupoids as special
  dagger Frobenius algebras in the category of sets and relations. 
  This is then generalized to a non-unital setting, by establishing an
  adjunction between H*-algebras in the 
  category of sets and relations, and locally cancellative regular
  semigroupoids. Finally, we study a universal passage from the former
  setting to the latter.
\end{abstract}

\section{Introduction}

Groupoids generalize groups in two ways. They can be regarded as
groups with more than one object, leading to the definition as (small)
categories in which every morphism is invertible. Alternatively, they can be
regarded as groups whose multiplication is relaxed to a partial
function. This article makes the connection between these two views
precise by detailing isomorphisms between the appropriate
categories. The latter view is made rigorous by so-called special
dagger Frobenius algebras in the category of sets and relations. Thus
we give a non-commutative and functorial generalization
of~\cite{pavlovic:frobrel} in Section~\ref{sec:frob}.

These results are somewhat surprising from the point of view of
quantum groups, another generalization of the concept of
group. Quantum groups
are usually formalized as some sort of Hopf algebra. However, this
notion is at odds with that of Frobenius algebra: if a multiplication
carries both Hopf and Frobenius structures, then it must be
trivial.\footnote{More precisely, in the language of
  Section~\ref{sec:frob}, if morphisms $m \colon X \otimes X \to X$
  and $m^\dag \colon X \to X \otimes X$ in a monoidal category
  satisfy (F), (A), (U), and the variant of (U) for $m^\dag$, and $m^\dag$ 
  and $u^\dag$ are monoid homomorphisms, then $X \cong I$.
  This result holds fully abstractly; for a proof in the category of
  finite-dimensional vector spaces, see~\cite[Proposition~2.4.10]{kock:frobenius}.}

Reformulating groupoids as relative Frobenius algebras has two
advantages. First, it yields several interesting new choices of
morphisms between groupoids. Second, Frobenius algebras
can be interpreted in many categories without limits, whereas the usual 
formulation of a groupoid as an internal
category requires the ambient category to have finite limits.  

For example, a commutative Frobenius algebra structure on a
finite-dimensional Hilbert space corresponds to a choice of
orthonormal basis for that space. For this correspondence to hold in
arbitrary dimension, the Frobenius algebra structure must be relaxed to
a so-called H*-algebra structure, basically dropping
units~\cite{abramskyheunen:hstar}. Section~\ref{sec:hstar} considers
the relative version: it turns out that H*-algebras in the category of sets
and relations correspond to so-called locally cancellative regular
semigroupoids. The correspondence is functorial, but this time 
gives adjunctions instead of isomorphisms of categories. Finally, 
Section~\ref{sec:quotient} considers a universal passage from
H*-algebras to Frobenius algebras.

The generalization to H*-algebras is useful for an application to
geometric quantization that will be presented in subsequent work,
where one is forced to work with semigroupoids instead of
groupoids. Rather than in the category of sets and relations, this
plays out in the smooth setting of symplectic manifolds and canonical
relations, corresponding to Lie groupoids. One could imagine similar
applications in a topological or localic setting~\cite{resende:etalegroupoids}.

\section{Relative Frobenius algebras and groupoids}
\label{sec:frob}

A (small) category can be defined as a category internal to the category $\Cat{Set}$
of sets and functions, see~\cite[Section~XII.1]{maclane:categories}. This is a
collection
\[\xymatrix@C+3ex{
    G_0 \ar|-{e}[r] 
    & G_1 \ar@<-1.2ex>|-{s}[l] \ar@<1.2ex>|-{t}[l] 
    & G_1 \times_{G_0} G_1 \ar|-{m}[l]
}\]
of objects $G_0$ (objects) and $G_1$ (morphisms) and morphisms $s$
(source), $t$ (target), $e$ (identity), and $m$ (composition). These data
have to satisfy familiar algebraic formulae, stating \textit{e.g.}\
that composition $m$ is associative. A functor then is a pair of
functions $f_i \colon G_i \to G'_i$ that commute with the above structure.
A category is a \emph{groupoid} when there additionally is a
morphism $i \colon 
G_1 \to G_1$ (inverse) satisfying $m \after (1 \times i) \after
\Delta = e \after s \colon G_1 \to G_1$. Notice that this formulation requires
the monoidal structure $\times$ to have diagonals $\Delta \colon G_1
\to G_1 \times G_1$, and the category to have pullbacks.

This section proves that groupoids correspond precisely to so-called
relative Frobenius algebras. To introduce the latter concept, we pass
to the category $\Cat{Rel}$ of sets and relations, where morphisms $X
\to Y$ are relations $r \subseteq X \times Y$, and
\[
  s \after r = \{ (x,z) \mid \exists y \,.\, (x,y) \in r, (y,z) \in s \}.
\]
It carries a contravariant identity-on-objects involution $\dag \colon
\Cat{Rel}\op \to \Cat{Rel}$ given by relational converse. It also has
compatible monoidal structure, namely Cartesian product of sets. This makes
$\Cat{Rel}$ into a so-called dagger symmetric monoidal category. 
To distinguish $\Cat{Rel}$ from its subcategory $\Cat{Set}$, we
write morphisms in the former category as $r \colon X \relto Y$, and
morphisms in the latter as $f \colon X \to Y$.

\begin{definition}
  Consider the following properties of a morphism $m \colon X \times X
  \relto X$ in $\Cat{Rel}$:
  \begin{align}
    \tag{F} (1 \times m) \after (m^\dag \times 1) = m^\dag \after m &
    = (m \times 1) \after (1
    \times m^\dag), \\
    \tag{M} m \after m^\dag & = 1, \\
    \tag{A} m \after (1 \times m) & = m \after (m \times 1), \\
    \tag{U} \text{ there is }u \colon 1 \relto X \text{ with } m
    \after (u \times 1) = 1 & = m \after (1 \times u).
  \end{align}
  If $u$ exists then it is automatically unique. 
  An object $X$ together with such a morphism $m$ is called a (unital)
  special dagger Frobenius algebra in $\Cat{Rel}$, or \emph{relative
  Frobenius algebra} for short. 
  Notice that this definition requires neither pullbacks nor diagonals,
  and makes sense in any dagger monoidal category.
\end{definition}

The defining conditions of Frobenius algebras can also be presented
graphically. Such string diagrams encode composition by drawing
morphisms on top of each other, and the monoidal product becomes 
drawing morphisms next to each other. The dagger becomes a vertical
reflection; we refer to~\cite{selinger:graphicallanguages} for more
information. We depict $m \colon X \times X \relto X$ as
$\smash{\vcenter{\hbox{\begin{tikzpicture}[scale=0.33,thick] 
  \node (3) at (-.75,-2) {};
  \node [style=dot] (1) at (0,-1) {};
  \node (0) at (0,0) {};
  \node (2) at (.75,-2) {};
  \draw [bend left=45] (1) to (2);
  \draw [bend left=45] (3) to (1);
  \draw (1) to (0);
\end{tikzpicture}}}}$, and $u \colon 1 \relto X$ as 
$\smash
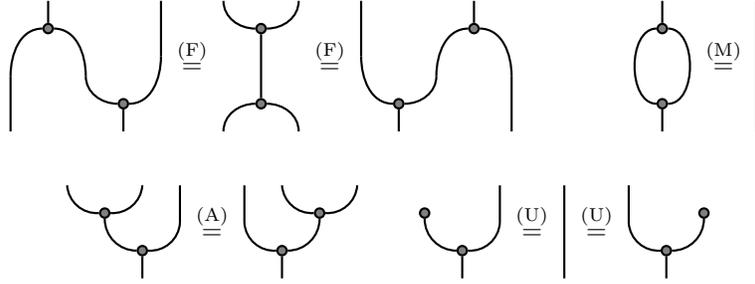
{{\hbox{\begin{tikzpicture}[scale=0.33,thick] 
  \node [style=dot] (1) at (0,-1) {};
  \node (0) at (0,0) {};
  \draw (1) to (0);
\end{tikzpicture}}}}$.
\[
  \vcenter{\hbox{\begin{tikzpicture}[thick,scale=0.5]
    \node (0) at (2, 2) {};
    \node (1) at (-1, 2) {};
    \node [style=dot] (2) at (-1, 1) {};
    \node (3) at (2, -0) {};
    \node (4) at (0, -0) {};
    \node (5) at (-2, -0) {};
    \node [style=dot] (6) at (1, -1) {};
    \node (7) at (1, -2) {};
    \node (8) at (-2, -2) {};
    \draw (2) to (1);
    \draw (6) to (7);
    \draw (5) to (8);
    \draw[in=90, out=180] (2) to (5.south);
    \draw[in=270, out=180] (6) to (4.south);
    \draw[in=0, out=270] (3.north) to (6);
    \draw[in=0, out=90] (4.south) to (2);
    \draw (0) to (3);
  \end{tikzpicture}}}
  \stackrel{\text{(F)}}{=}
  \vcenter{\hbox{\begin{tikzpicture}[thick,scale=0.5]
    \node (0) at (-1, 2) {};
    \node (1) at (1, 2) {};
    \node [style=dot] (2) at (0, 1) {};
    \node [style=dot] (3) at (0, -1) {};
    \node (4) at (-1, -2) {};
    \node (5) at (1, -2) {};
    \draw (2) to (3);
    \draw[in=90, out=0] (3) to (5);
    \draw[in=180, out=90] (4) to (3);
    \draw[in=180, out=270] (0) to (2);
    \draw[in=270, out=0] (2) to (1);
  \end{tikzpicture}}}
  \stackrel{\text{(F)}}{=}
  \vcenter{\hbox{\begin{tikzpicture}[thick,scale=0.5]
    \node (0) at (-2, 2) {};
    \node (1) at (1, 2) {};
    \node [style=dot] (2) at (1, 1) {};
    \node (3) at (-2, -0) {};
    \node (4) at (0, -0) {};
    \node (5) at (2, -0) {};
    \node [style=dot] (6) at (-1, -1) {};
    \node (7) at (-1, -2) {};
    \node (8) at (2, -2) {};
    \draw (2) to (1);
    \draw (6) to (7);
    \draw (5) to (8);
    \draw[in=90, out=0] (2) to (5.south);
    \draw[in=270, out=0] (6) to (4.south);
    \draw[in=180, out=270] (3.north) to (6);
    \draw[in=180, out=90] (4.south) to (2);
    \draw (0) to (3);
  \end{tikzpicture}}}
  \qquad\qquad
  \vcenter{\hbox{\begin{tikzpicture}[thick,scale=0.5]
    \node (0) at (0,-2) {};
    \node (3) at (0,2) {};
    \node[dot] (2) at (0,1) {};
    \node[dot] (1) at (0,-1) {};
    \draw [bend right=90]  (2) to (1);
    \draw  (2) to (3);
    \draw [bend right=90]  (1) to (2);
    \draw  (0) to (1);
  \end{tikzpicture}}}
  \stackrel{\text{(M)}}{=}
  \vcenter{\hbox{\begin{tikzpicture}[thick,scale=0.5]
    \node (0) at (0,-2) {};
    \node (3) at (0,2) {};
    \draw  (0) to (3);
  \end{tikzpicture}}}
\]
\[
   \vcenter{\hbox{\begin{tikzpicture}[thick,scale=0.5]
    \node [style=dot] (1) at (0,-1) {};
    \node (5) at (0,1) {};
    \node (0) at (0,-2) {};
    \node (6) at (1,1) {};
    \node (2) at (1,0) {};
    \node (4) at (-2,1) {};
    \node [style=dot] (3) at (-1,0) {};
    \draw (1) to (0);
    \draw [bend right=45]  (1) to (2);
    \draw [bend left=45]  (5) to (3);
    \draw (6) to (2.south);
    \draw [bend right=45]  (4) to (3);
    \draw [bend right=45]  (3) to (1);
  \end{tikzpicture}}}
  \stackrel{\text{(A)}}{=}
  \vcenter{\hbox{\begin{tikzpicture}[thick,scale=0.5]
    \node [style=dot] (1) at (0,-1) {};
    \node (5) at (0,1) {};
    \node (0) at (0,-2) {};
    \node (6) at (-1,1) {};
    \node (2) at (-1,0) {};
    \node (4) at (2,1) {};
    \node [style=dot] (3) at (1,0) {};
    \draw  (1) to (0);
    \draw [bend left=45]  (1) to (2);
    \draw [bend right=45]  (5) to (3);
    \draw  (6) to (2.south);
    \draw [bend left=45]  (4) to (3);
    \draw [bend left=45]  (3) to (1);
  \end{tikzpicture}}}
  \qquad
  \vcenter{\hbox{\begin{tikzpicture}[thick,scale=0.5]
    \node [style=dot] (3) at (-1,2) {};
    \node [style=dot] (1) at (0,1) {};
    \node (0) at (0,0) {};
    \node (2) at (1,2) {};
    \node (4) at (1,3) {};
    \draw [bend right=45]  (1) to (2);
    \draw [bend right=45]  (3) to (1);
    \draw (1) to (0);
    \draw (2.south) to (4);
  \end{tikzpicture}}}
  \stackrel{\text{(U)}}{=}
  \vcenter{\hbox{\begin{tikzpicture}[thick,scale=0.5]
    \node (0) at (0,-1) {};
    \node (3) at (0,2) {};
    \draw  (0) to (3);
  \end{tikzpicture}}}
  \stackrel{\text{(U)}}{=}
  \vcenter{\hbox{\begin{tikzpicture}[thick,scale=0.5]
    \node [style=dot] (3) at (1,2) {};
    \node [style=dot] (1) at (0,1) {};
    \node (0) at (0,0) {};
    \node (2) at (-1,2) {};
    \node (4) at (-1,3) {};
    \draw [bend left=45]  (1) to (2);
    \draw [bend left=45]  (3) to (1);
    \draw (1) to (0);
    \draw (2.south) to (4);
  \end{tikzpicture}}}
\]
To prevent jumping back and forth between formalisms, we will mostly
compute algebraically. But occasionally we will illustrate
conditions graphically.

\subsection{From relative Frobenius algebras to groupoids}

For the rest of this subsection, fix a relative Frobenius algebra $(X,m)$.

Concretely, (M) implies that $m$ is single-valued. Therefore we may
write $f=hg$ instead of $((h,g),f) \in m$. Notice that this notation
implies that $hg$ is defined, which fact we denote by $hg
\defined$. We will use this \emph{Kleene equality} throughout this
article, by reading $x=y$ for $x,y \in X$ as follows: if either side is
defined, so is the other, and they are equal. The concrete meaning of
(M) is completed by   
\[
 \forall f \in X \,\exists g,h \in X .\, f=hg.
\]
Concretely, (F) means that for all $a,b,c,d \in X$
\[
  ab=cd \iff \exists e \in X .\, b=ed, c=ae \iff \exists e \in X .\, d=eb, a=ce.
\]
The condition (A) comes down to $(fg)h=f(gh)$. Finally, identifying
the morphism $u \colon 1 \relto X$ with a subset $U \subseteq X$, we
find that (U) means 
\begin{align*}
  \forall f \in X \, \exists u \in U .&\, fu=f, \\
  \forall f \in X \, \exists u \in U .&\, uf=f, \\
  \forall f \in X \, \forall u \in U .&\, fu \defined \implies fu=f, \\
  \forall f \in X \, \forall u \in U .&\, uf \defined \implies uf=f.
\end{align*}

\begin{definition}
  Given a relative Frobenius algebra $m$, define the following objects and morphisms in
  $\Cat{Rel}$:
  \begin{align*}
    G_1 & = X, \\
    G_2 & = \{(g,f) \in X^2 \mid gf \defined \}, \\
    G_0 & = U, \\
    e & = U \times U \colon G_0 \relto G_1, \\
    s & = \{ (f,x) \in G_1 \times G_0 \mid fx \defined \} \colon G_1 \relto G_0, \\
    t & = \{ (f,y) \in G_1 \times G_0 \mid yf \defined \} \colon G_1 \relto G_0, \\
    i & = \{ (g,f) \in G_2 \mid gf \in G_0, fg \in G_0 \} \colon G_1 \relto G_1.
  \end{align*}
\end{definition}

We will prove that the collection $\cat{G}$ of these data is a
groupoid (in $\Cat{Set}$). First, we show that the relations $s,t,i$ are in fact
(graphs of) functions, as is clearly the case for $e$. 

\begin{lemma}
  For $f \in X$ and $u,v \in U$:
  \begin{enumerate}
  \item if $fu \defined$ then $u^2 \defined$;
  \item if $fu \defined$ and $fv \defined$ then $uv \defined$;
  \item if $fu \defined$ and $fv \defined$ then $u=v$.
  \end{enumerate}
  Hence the relation $s$ is (the graph of) a function. Similarly, $t$ is (the graph of) a
  function.
\end{lemma}
\begin{proof}
  If $fu \defined$, then $fu=f$ by (U), so that also $(fu)u=f$. By (A), this means in particular
  that $u^2\defined$, establishing (1). For part (2), assume that $fu \defined$ and $fv
  \defined$. Then $fu=f=fv$, and by (F) we have $u=ev$ for some $e \in X$, so that
  $uv=ev^2 \defined$. Finally, for (3), notice that if $fu \defined$ and $fv \defined$
  then $u=uv=v$ by (U).
\end{proof}

\begin{lemma}
  The pullback of $s$ and $t$ in $\Cat{Set}$ is (isomorphic to) $G_2$.
\end{lemma}
\begin{proof}
  The pullback of $s$ and $t$ is given by $P = \{ (g,f) \in X \mid s(g)=t(f) \}$. 
  Now $s(g)$ is the unique $y \in U$ with $gy \defined$, and $t(f)$ is the unique $y'
  \in U$ with $y'f \defined$. So, if $(g,f) \in P$, then $y=y'$ so that $gyf \defined$, and by
  (U) also $gf \defined$ so $(g,f) \in G_2$. Conversely, if $(g,f) \in G_2$, then by (U) there
  exists $y \in U$ such that $gyf \defined$, and we have $s(g)=y=t(f)$.
\end{proof}

\begin{lemma}\label{lem:inverses}
  The following diagram in $\Cat{Rel}$ commutes.
  \[\xymatrix{
    G_1 \ar_-{\Delta}[d] \ar^-{s}[rr] && G_0 \ar^-{e}[d] \\
    G_1 \times G_1 \ar_-{1 \times i}[r] & G_1 \times G_1 \ar_-{m}[r] & G_1
  }\]
  Here, $\Delta$ is (the graph of) the diagonal function $x \mapsto (x,x)$.
\end{lemma}
\begin{proof}
  First we compute 
  \[
    e \after s 
    = \{(f,g) \in G_1 \times G_1 \mid \exists u \in U .\, g=u, fu \defined \}
    = \{ (f,u) \in X \times U \mid fu \defined \},
  \]
  and
  \begin{align*}
    m \after (1 \times i) \after \Delta 
    & = \{ (f,h) \in G_1^2 \mid \exists g \in G_1 .\, fg \in U, gf \in U, h=gf \} \\
    & = \{ (f,gf) \in G_1^2 \mid g \in G_1 ,\, fg \in U \ni gf \}.
  \end{align*}
  Clearly $m \after (1 \times i) \after \Delta \subseteq e \after s$. For the converse,
  suppose that $(f,u) \in e \after s$. Since $fu \defined$ we then have $fu=f$ by (U).
  Therefore, again by (U), there exists $v \in U$ such that $fu=vf$. Now it follows from (F)
  that there exists $g$ with $u=gf$ and $v=fg$. Thus $(f,u)=(f,gf) \in m \after (1 \times i)
  \after \Delta$.
\end{proof}

\begin{lemma}
  The relation $i$ is (the graph of) a function.    
\end{lemma}
\begin{proof}
  We need to prove that to each $f \in X$ there is a unique $g \in X$ with $gf \in U \ni fg$.
  We already have existence of such a $g$ by Lemma~\ref{lem:inverses}, so it suffices to
  prove unicity. Suppose that $gf \in U \ni fg$ and $g'f \in U \ni fg'$. Then in particular $fg
  \defined$ and $gf \defined$, so that by (A) also $fgf \defined$ and similarly $fg'f
  \defined$. Now (U) implies $fgf=f=fg'f$, so that by the previous conjecture $gf=g'f$. But
  then $g=gfg=g'fg=g'$.  
\end{proof}

\begin{theorem}\label{frobisgpd}
  If $m$ is a relative Frobenius algebra, then $\cat{G}$ is a groupoid
  (in $\Cat{Set}$).
\end{theorem}
\begin{proof}
  The proof consists of routine verifications that the maps $e,s,t,i$ indeed satisfy the
  axioms for a groupoid. The most interesting one has already been
  dealt with in Lemma~\ref{lem:inverses}.
\end{proof}

\subsection{From groupoids to relative Frobenius algebras}

For the rest of this subsection, fix a groupoid 
\[
  \cat{G}=\left(\xymatrix@C+3ex{
    G_0 \ar|-{e}[r] 
    & G_1 \ar@(ul,ur)|-{i} \ar@<-1.2ex>|-{s}[l] \ar@<1.2ex>|-{t}[l] 
    & G_1 \times_{G_0} G_1 \ar|-{m}[l]
  }\right).
\]

\begin{definition}\label{def:gpdtofrob}
  For a groupoid $\cat{G}$, define $X=G_1$, and let $m \colon G_1 \times G_1 \relto G_1$ be
  the graph of the function $m$.
\end{definition}

We will prove that $m$ is a relative Frobenius algebra. For starters, it follows directly from associativity of composition in the groupoid $\cat{G}$ that $m$ satisfies (A).

\begin{lemma}
  The morphism $m$ of $\Cat{Rel}$ satisfies (U).
\end{lemma}
\begin{proof}
  Define a relation $u \colon 1 \relto X$ by $u = \{ (*,e(x)) \mid x \in G_0 \}$.  Then
  \begin{align*}
    m \after (u \times 1) 
    & = \{ (f,e(x)f) \mid f \in G_1, x=t(f) \in G_0 \} \\
    & = \{ (f,et(f)f) \mid f \in G_1 \} = 1.
 \end{align*}
  The symmetric condition also holds, and so (U) is satisfied.
\end{proof}

\begin{lemma}
  The morphism $m$ of $\Cat{Rel}$ satisfies (M).
\end{lemma}
\begin{proof}
  We have
  \[
    m \after m^\dag = \{ (f,f) \in G_1^2 \mid \exists g,h \in G_2 .\, s(h)=t(g) ,\, f=hg \}.
  \]
  Because we can always take $g=f$ and $h=e(t(f))$, this relation is equal to $\{(f,f) \in
  G_1^2 \mid f \in G_1\}=1$. Thus (M) is satisfied.
\end{proof}

\begin{lemma}
  The morphism $m$ of $\Cat{Rel}$ satisfies (F).
\end{lemma}
\begin{proof}
  First compute
  \begin{align*}
    m^\dag \after m & = \{ ((a,b),(c,d)) \in G_2^2 \mid ab=cd \}, \\
    (m \times 1) \after (1 \times m^\dag) & = \{ ((a,b),(c,d)) \in G_2^2 \mid \exists e \in G_1 .\, ed=b,\, ae=c \}.
 \end{align*}
  If $ed=b$ and $ae=c$, then $cd=aed=ab$. Hence $(m \times 1) \after (1 \times m^\dag)
  \subseteq m^\dag \after m$. Conversely, suppose that $((a,b),(c,d)) \in m^\dag \after m$.
  Taking $e=bd^{-1}$, then $ed=bdd^{-1}=b$, and $ae=abd^{-1}=cdd^{-1}=c$. Therefore
  $m^\dag \after m \subseteq (1 \times m^\dag) \after (m \times 1)$. The symmetric
  condition is verified analogously. Thus (F) is satisfied.
\end{proof}

\begin{theorem}\label{gpdisfrob}
  If $\cat{G}$ is a groupoid, then $m$ is a relative Frobenius algebra.
  \qed
\end{theorem}

\subsection{Functoriality}\label{subsec:frob:functoriality}

Notice that the constructions $m \mapsto \cat{G}$ and $\cat{G} \mapsto
m$ of the previous two sections are each other's inverse. This subsection
proves that the assignments extend to an isomorphism of
categories under various choices of morphisms: one that is natural for
groupoids, one that is natural for relations, and one that is natural
to Frobenius algebras. 
(See also~\cite{coeckepaquettepavlovic:structuralism}.)
We start by considering a choice of morphisms that is natural from the
point of view of relations: namely, morphisms between groupoids are
subgroupoids of the product.

The category $\cat{Rel}$ is \emph{compact closed}, \textit{i.e.}\ 
allows morphisms $\eta_X \colon 1 \to X \times X$ satisfying $(\eta^\dag
\times 1) \after (1 \times \eta) = 1 = (1 \times \eta^\dag) \after
(\eta \times 1)$. Drawing $\eta$ as $\smash{\hbox{\begin{tikzpicture}[scale=0.33]
     \node (0) at (-.5,.5) {};
     \node (1) at (.5,.5) {};
     \draw [thick, in=270, out=270, looseness=1.5] (0) to (1);
  \end{tikzpicture}}}$, this property graphically becomes the following.
\[
   \vcenter{\hbox{\begin{tikzpicture}[scale=0.5]
		\node [style=none] (0) at (-1, 1) {};
		\node [style=none] (1) at (-1, 0) {};
		\node [style=none] (2) at (0, 0) {};
		\node [style=none] (3) at (1, 0) {};
		\node [style=none] (4) at (1, -1) {};
		\draw [thick, bend right=90, looseness=1.75] (1.center) to (2.center);
		\draw [thick] (0.center) to (1.center);
		\draw [thick] (3.center) to (4.center);
		\draw [thick, bend left=90, looseness=1.75] (2.center) to (3.center);
   \end{tikzpicture}}}
   \; = \;
   \vcenter{\hbox{\begin{tikzpicture}[scale=0.5]
		\node [style=none] (0) at (0, 1) {};
		\node [style=none] (1) at (0, -1) {};
		\draw [thick] (0.center) to (1.center);
   \end{tikzpicture}}}
   \; = \;
   \vcenter{\hbox{\begin{tikzpicture}[scale=0.5]
		\node [style=none] (0) at (1, 1) {};
		\node [style=none] (1) at (1, 0) {};
		\node [style=none] (2) at (0, 0) {};
		\node [style=none] (3) at (-1, 0) {};
		\node [style=none] (4) at (-1, -1) {};
		\draw [thick, bend left=90, looseness=1.75] (1.center) to (2.center);
		\draw [thick] (0.center) to (1.center);
		\draw [thick] (3.center) to (4.center);
		\draw [thick, bend right=90, looseness=1.75] (2.center) to (3.center);
   \end{tikzpicture}}}
\]
In fact, any relative Frobenius algebra induces such a compact
structure on $X$, by $\eta = m^\dag \after u$. There is a canonical 
choice of compact structure, that we use from now on, namely
$\eta = \{(*,(x,x)) \mid x \in X\}$. It is induced by the relative
Frobenius algebra corresponding to the discrete groupoid on $X$, and
therefore has a universal categorical characterization. It follows
that morphisms $r \colon X \relto Y$ in $\cat{Rel}$ have 
transposes $\name{r} = (1 \times r) \after \eta = \{(*,(x,y)) \mid (x,y) \in 
r\} \colon 1 \relto X \times Y$. Furthermore, the dagger in $\cat{Rel}$ is compatible with 
the symmetric monoidal structure, giving a natural swap isomorphism
$\swapmap \colon X \times Y \to Y \times X$ with
$\swapmap^{-1}=\swapmap^\dag$. 

\begin{definition}\label{def:relmorphisms}
  The category $\Cat{Frob}(\Cat{Rel})\rel$ has relative Frobenius
  algebras as objects. A morphism $(X,m_X) \to (Y,m_Y)$ is a morphism
  $r \colon X \relto Y$ in $\Cat{Rel}$ satisfying
  \begin{align}
    (m_X \times m_Y) \after (1 \times \swapmap \times 1) \after
    (\name{r} \times \name{r}) & = \name{r}, \tag{R} \\
    (r \times \eta^\dag) \after (m_X^\dag \times 1) \after (u_X \times 1)
    & = (u_Y^\dag \times 1) \after (m_Y \times 1) \after (r \times \eta)
    \tag{I}
  \end{align}
  These conditions translate into string diagrams as follows.
  \[
    \vcenter{\hbox{\begin{tikzpicture}[scale=0.5]
        \begin{pgfonlayer}{nodelayer}
        	\node [style=none] (0) at (0, 1) {};
        	\node [style=none] (1) at (1, 1) {};
        	\node [style=none] (2) at (0, 0) {};
        	\node [style=none] (3) at (1, 0) {};
                \node [thick, style=box] (4) at (1,.3) {$r$};
        \end{pgfonlayer}
        \begin{pgfonlayer}{edgelayer}
        	\draw [thick, bend left=90, looseness=2.00] (3.north) to (2.north);
        	\draw [thick] (1) to (3);
        	\draw [thick] (0) to (2);
        \end{pgfonlayer}
      \end{tikzpicture}}}
      \;\; \stackrel{\text{(R)}}{=} \;\;
      \vcenter{\hbox{\begin{tikzpicture}[scale=0.5]
        \begin{pgfonlayer}{nodelayer}
        	\node [style=none] (0) at (-0.5, 2.5) {};
        	\node [style=none] (1) at (1.5, 2.5) {};
        	\node [thick, style=dot] (2) at (-0.5, 1.75) {};
        	\node [thick, style=dot] (3) at (1.5, 1.75) {};
        	\node [style=none] (4) at (-1, 1) {};
        	\node [style=none] (5) at (0, 1) {};
        	\node [style=none] (6) at (1, 1) {};
        	\node [style=none] (7) at (2, 1) {};
        	\node [style=none] (8) at (-1, 0) {};
        	\node [style=none] (9) at (0, 0) {};
        	\node [style=none] (10) at (1, 0) {};
        	\node [style=none] (11) at (2, 0) {};
        	\node [thick, style=box] (12) at (0, -0.3) {$r$};
        	\node [thick, style=box] (13) at (2, -0.3) {$r$};
        	\node [style=none] (14) at (-1, -.5) {};
        	\node [style=none] (15) at (0, -.5) {};
        	\node [style=none] (16) at (1, -.5) {};
        	\node [style=none] (17) at (2, -.5) {};
        \end{pgfonlayer}
        \begin{pgfonlayer}{edgelayer}
        	\draw [thick, bend left=90, looseness=2.00] (17.center) to (16.center);
        	\draw [thick] (10.center) to (16.center);
        	\draw [thick] (7.center) to (11.center);
        	\draw [thick] (0.center) to (2.center);
        	\draw [thick] (3.center) to (1.center);
        	\draw [thick, in=180, out=90] (6.center) to (3.center);
        	\draw [thick, in=180, out=90] (4.center) to (2.center);
        	\draw [thick, in=90, out=270] (11.center) to (17.center);
        	\draw [thick, in=90, out=0] (2.center) to (5.center);
        	\draw [thick, bend left=90, looseness=2.00] (15.center) to (14.center);
        	\draw [thick, bend left=45] (3.center) to (7.center);
        	\draw [thick] (8.center) to (14.center);
        	\draw [thick] (9.center) to (15.center);
        	\draw [thick, in=90, out=270] (6.center) to (9.center);
        	\draw [thick] (8.center) to (4.center);
        	\draw [thick, cross, in=90, out=-90] (5.center) to (10.center);
        \end{pgfonlayer}
      \end{tikzpicture}}}
      \qquad \qquad
      \vcenter{\hbox{\begin{tikzpicture}[thick, scale=0.33]
		\node [style=none] (0) at (-1, 2) {};
		\node [style=box] (1) at (-1, 1) {$r$};
		\node [style=none] (2) at (2, 1) {};
		\node [style=none] (3) at (-1, 0) {};
		\node [style=none] (4) at (1, 0) {};
		\node [style=none] (5) at (3, 0) {};
		\node [style=dot] (6) at (0, -1) {};
		\node [style=dot] (7) at (0, -2) {};
		\node [style=none] (8) at (3, -2) {};
		\draw (6) to (7);
		\draw [bend left=45, looseness=1.25] (2.center) to (5.center);
		\draw (5.center) to (8);
		\draw (0) to (1);
		\draw [bend right=45, looseness=1.25] (3.center) to (6.center);
		\draw [bend right=45, looseness=1.25] (6.center) to (4.center);
		\draw (1) to (3.center);
		\draw [bend left=45, looseness=1.25] (4.center) to (2.center);
		\node [style=dot] (6) at (0, -1) {};
      \end{tikzpicture}}}
      \; \stackrel{\text{(I)}}{=} \;
      \vcenter{\hbox{\begin{tikzpicture}[thick, scale=0.33]
		\node [style=none] (0) at (-1, -2) {};
		\node [style=box] (1) at (-1, -1) {$r$};
		\node [style=none] (2) at (2, -1) {};
		\node [style=none] (3) at (-1, 0) {};
		\node [style=none] (4) at (1, 0) {};
		\node [style=none] (5) at (3, 0) {};
		\node [style=dot] (6) at (0, 1) {};
		\node [style=dot] (7) at (0, 2) {};
		\node [style=none] (8) at (3, 2) {};
		\draw (6) to (7);
		\draw [bend right=45, looseness=1.25] (2.center) to (5.center);
		\draw (5.center) to (8);
		\draw (0) to (1);
		\draw [bend left=45, looseness=1.25] (3.center) to (6.center);
		\draw [bend left=45, looseness=1.25] (6.center) to (4.center);
		\draw (1) to (3.center);
		\draw [bend right=45, looseness=1.25] (4.center) to (2.center);
		\node [style=dot] (6) at (0, 1) {};
      \end{tikzpicture}}}
  \]
\end{definition}

\begin{proposition}
  $\cat{Frob}(\cat{Rel})\rel$ is a well-defined category.
\end{proposition}
\begin{proof}
  Clearly, identities $1_X = \{(x,x) \mid x \in X \} \colon X \relto X$
  satisfy (R) and (I). Observe that the composition $(1 \times \eta^\dag) \after
  (m^\dag \times 1) \after (u \times 1) \colon X \relto X$ is the
  relation $\{(x,x^{-1}) \mid x \in X\}$, where $x^{-1}$ is the
  inverse of $x$ when regarding $X$ as the set of morphisms of a
  groupoid as per Theorem~\ref{frobisgpd}. So (I) means that $(x,y) \in
  r$ if and only if $(x^{-1}, y^{-1}) \in r$. Now, if $r \colon X
  \relto Y$ and $s \colon Y \relto Z$ satisfy (R) and (I), then so
  does $s \after r$: 
  \begin{align*}
    \name{s \after r}
    & = \{(*, x'',z'') \in 1 \times X \times Z \mid \exists y'' \in Y
    .\, (x'',y'') \in r, (y'',z'') \in s \} \\
    & = \{(*, xx', zz') \mid x,x' \in X, z,z' \in Z, \exists y,y' \in Y . \\
    & \qquad\qquad\qquad\quad\; (x,y) \in r, (x',y') \in r, (y,z) \in s, (y',z') \in s \} \\
    & = (m_X \times m_Z) \after (1 \times \swapmap \times 1) \after
    (\name{s \after r} \times \name{s \after r}).
  \end{align*}
  Let us justify the second equation. If $(x'',y'') \in r$, then
  $(x''^{-1},y''^{-1}) \in r$ by (I), and hence $(1,1) \in r$ by
  (R). Hence we may take $x=x''$, $x'=1$, $y=y''$ and $y'=1$.
\end{proof}

\begin{definition}
  The category $\Cat{Gpd}\rel$ has groupoids as objects. Morphisms $\cat{G}
  \to \cat{H}$ are subgroupoids of $\cat{G} \times \cat{H}$.
\end{definition}

That this is a well-defined category will follow from the following
theorem. Identities are the diagonal subgroupoids, and composition of
subgroupoids $\cat{R} \subseteq \cat{G} \times \cat{G}'$ and $\cat{S}
\subseteq \cat{G}' \times \cat{G}''$ is the groupoid $S_1 \after R_1
\rightrightarrows S_0 \after R_0$.

\begin{theorem}\label{gpdfrobext}
  There is an isomorphism of categories $\Cat{Frob}(\Cat{Rel})\rel
  \cong \Cat{Gpd}\rel$.
\end{theorem}
\begin{proof}
  Let $(X,m_X)$ and $(Y,m_y)$ be relative Frobenius algebras, inducing
  groupoids $\cat{G}$ and $\cat{H}$.
  First, notice that if $r \colon X \relto Y$ satisfies (R), then
  \begin{align*}
    m_r & = (m_X \times m_Y) \after (1 \times \swapmap \times 1) \\
    & = \{ (((a,b),(c,d)),(ac,bd)) \mid a,b,c,d \in X \}
    \colon r \times r \relto r
  \end{align*}
  is a well-defined morphism in $\Cat{Rel}$. In fact, since $(X,m_X)$ and
  $(Y,m_Y)$ are relative Frobenius algebras, so is $(r,m_r)$: one
  readily verifies that it satisfies (M), (A), and (F). Also, (U) is
  satisfied by the intersection $1 \relto R$ of  $r$ with $U_X \times
  U_Y$. Theorem~\ref{frobisgpd} thus shows that $r$ induces a groupoid
  $\cat{R}$. It is a subgroupoid of $\cat{G} \times \cat{H}$: we have
  $R_1 \subseteq (G \times H)_1$ by construction, and if $u \in U_R$,
  then $u=u^{-1}$, so $u \in (G \times H)_0$. The structure maps
  of $\cat{R}$ are easily seen to be restrictions of those of $\cat{G}
  \times \cat{H}$.

  Conversely, if $\cat{R}$ is a subgroupoid of $\cat{G} \times
  \cat{H}$, then clearly $R_1 \subseteq X \times Y$ is a morphism in
  $\cat{Rel}$ satisfying (R) and (I). It now suffices to observe that these
  constructions are inverses.
\end{proof}

Next, we consider a choice of morphisms that is natural to groupoids,
namely functors. This entails dealing with functions. Fortunately,
functions can be characterized among relations purely
categorically. The category $\cat{Rel}$ is in fact a 2-category, 
where there is a single 2-cell $r \Rightarrow s$ when $r \subseteq
s$. Hence it makes sense to speak of adjoints of morphisms in
$\Cat{Rel}$. A morphism has a right adjoint if and only if it is (the
graph of) a function. 

\begin{definition}
  The category $\Cat{Frob}(\Cat{Rel})$ has relative Frobenius algebras
  as objects. Morphisms $(X,m_X) \to (Y,m_Y)$ are morphisms $r \colon
  X \relto Y$ that satisfy (I) and preserve the multiplication: $r \after m_X = m_Y
  \after (r \times r)$. 
  \[\vcenter{\hbox{\begin{tikzpicture}[thick, scale=0.5]
    \node (0) at (0,2) {};
    \node [box] (1) at (0,1) {$r$};
    \node [dot] (2) at (0,0) {};
    \node (3) at (-.75,-1) {};
    \node (4) at (.75,-1) {};
    \draw (0) to (1) to (2);
    \draw [out=180, in=90] (2) to (3);
    \draw [out=0, in=90] (2) to (4);
  \end{tikzpicture}}}
  =
  \vcenter{\hbox{\begin{tikzpicture}[thick, scale=0.5]
    \node (0) at (0,1) {};
    \node [dot] (2) at (0,0) {};
    \node [box] (3) at (-.75,-1) {$r$};
    \node [box] (4) at (.75,-1) {$r$};
    \node (5) at (-.75,-2) {};
    \node (6) at (.75,-2) {};
    \draw (0) to (2);
    \draw [out=180, in=90] (2) to (3);
    \draw [out=0, in=90] (2) to (4);
    \draw (3) to (5);
    \draw (4) to (6);
  \end{tikzpicture}}}\]
  We write $\Cat{Frob}(\Cat{Rel})\func$ for the subcategory of all
  morphisms $r$ that have a right adjoint and allow a 2-cell $r \after
  u_X \Rightarrow u_Y$.
\end{definition}

\begin{lemma}
  Morphisms in $\Cat{Frob}(\Cat{Rel})$ satisfy (R). Hence we have
  inclusions
  \[
    \Cat{Frob}(\Cat{Rel})\func
    \hookrightarrow
    \Cat{Frob}(\Cat{Rel})
    \hookrightarrow
    \Cat{Frob}(\Cat{Rel})\rel.
  \]
\end{lemma}
\begin{proof}
  Let $r \colon X \relto Y$ be a morphism in $\Cat{Frob}(\Cat{Rel})$. Then
  \begin{align*}
    & (m_X \times m_Y) \after (1 \times \swapmap \times 1) \after
    (\name{r} \times \name{r}) \\
    & = (m_X \times m_Y) \after (1 \times r \times r) \after (1 \times
    \swapmap \times 1) \after \{ (*, (x,x,y,y)) \mid x,y \in X \} \\
    & = (1 \times r) \after (m_X \times m_X) \after \{(*,(x,y,x,y))
    \mid x,y \in X \} \\
    & = (1 \times r) \after \{(*,(xy,xy)) \mid xy \defined \} \\
    & = (1 \times r) \after \{(*,(z,z)) \mid z \in X \} \\
    & = \name{r},
  \end{align*}
  because we can always choose $x=z$ and $y=1$.
\end{proof}

Write $\Cat{Gpd}$ for the category of groupoids and functors.

\begin{theorem}\label{gpdfrobfunc}
  There is an isomorphism of categories $\Cat{Frob}(\Cat{Rel})\func \cong \Cat{Gpd}$.
\end{theorem}
\begin{proof}
  Let $(X,m_X)$ and $(Y,m_Y)$ be relative Frobenius algebras, inducing
  groupoids $\cat{G}$ and $\cat{H}$. Let $r \colon m_X \to m_Y$ be a
  morphism in $\Cat{Frob}(\Cat{Rel})\func$.
  The condition that $r$ has a right adjoint means it is in fact a
  function $r \colon G_1 \to H_1$. Furthermore, the condition $r
  \after u_X \subseteq u_Y$ means precisely that it sends $G_0$ to
  $H_0$. Finally, the condition that $r$ preserve multiplication makes
  it functorial $\cat{G} \to \cat{H}$, because relational composition
  of graphs coincides with composition of functions.
  Conversely, it
  is easy to see that a functor between groupoids induces a morphism
  in $\Cat{Frob}(\Cat{Rel})\func$. Finally, these two constructions
  are inverse to each other. 
\end{proof}

\begin{corollary}
  The category $\Cat{Gp}$ of groups and homomorphisms is
  isomorphic to the full subcategory of $\Cat{Frob}(\Cat{Rel})\func$
  consisting of those relative Frobenius algebras for which $u$ has
  a right adjoint.
\end{corollary}
\begin{proof}
  The morphism $u \colon 1 \relto U$ has a right adjoint precisely
  when it is a function $u \colon 1 \to U$ and hence amounts to an
  element of $U$. That is, the corresponding groupoid has a 
  single identity, \textit{i.e.} is a group.
\end{proof}

Finally, we can consider a choice of morphisms that is natural from
the point of view of Frobenius algebras, namely the category
$\Cat{Frob}(\Cat{Rel})$. 
There is a category between $\Cat{Gpd}$ and $\Cat{Gpd}\rel$, 
that corresponds to the middle category in the sequence
$
    \Cat{Frob}(\Cat{Rel})\func
    \hookrightarrow
    \Cat{Frob}(\Cat{Rel})
    \hookrightarrow
    \Cat{Frob}(\Cat{Rel})\rel
$, 
as follows.

\begin{definition}\label{def:mfunc}
  A \emph{multi-valued functor} $\cat{G} \to \cat{H}$ between
  categories is a multi-valued function $F \colon G_1 \to H_1$ that
  preserves identities and composition: 
  \begin{align*}
    &\mbox{for } g,f \in G_1 \times_{G_0} G_1:
    && g \after f \ni h \;\Rightarrow\; F(g) \after F(f) \ni F(h), \\
    &\mbox{for } x \in G_0: && F(e(x)) \ni H_0.
  \end{align*}
  We denote the category of groupoids and multi-valued functors by
  $\Cat{Gpd}\mfunc$.  
\end{definition}

\begin{theorem}\label{grpfrobmfunc}
  There is an isomorphism of categories $\Cat{Frob}(\Cat{Rel}) \cong \Cat{Gpd}\mfunc$.
\end{theorem}
\begin{proof}
  Let $m_X$ and $m_Y$ be relative Frobenius algebras, inducing
  groupoids $\cat{G}$ and $\cat{H}$. Let $r \colon m_X \to m_Y$ be a
  morphism in $\Cat{Frob}(\Cat{Rel})$; by Theorem~\ref{gpdfrobext} it
  induces a subgroupoid of $\cat{G} \times \cat{H}$. By the
  argument of the proof of Theorem~\ref{gpdfrobfunc}, $r$ is a
  multi-valued function $G_1 \to H_1$. But then it is precisely a
  multi-valued functor.
\end{proof}

In summary, we have the following commutative diagram of categories.
\[\xymatrix@C+2ex@R-2ex{
  \Cat{Frob}(\Cat{Rel})\func \ar^-{\cong}[d] \ar@{^{(}->}[r]
  & \Cat{Frob}(\Cat{Rel}) \ar^-{\cong}[d] \ar@{^{(}->}[r]
  & \Cat{Frob}(\Cat{Rel})\rel \ar^-{\cong}[d] \\
  \Cat{Gpd} \ar@{^{(}->}[r]
  & \Cat{Gpd}\mfunc \ar@{^{(}->}[r]
  & \Cat{Gpd}\rel
}\]

\section{Relative H*-algebras and semigroupoids}
\label{sec:hstar}

This section establishes a non-unital generalization of the
correspondence of the previous section. Frobenius algebras are relaxed
to the following non-unital version.

\begin{definition}
  A \emph{relative H*-algebra} is a morphism $m \colon X \times X
  \relto X$ in $\Cat{Rel}$ satisfying (M), (A), and
  \[
    \tag{H} \begin{array}{c} \text{ there is an involution }* \colon
      \Cat{Rel}(1,X) \to \Cat{Rel}(1,X) \text{ such that } \\
      m \after (1 \times x^*) = (1 \times x^\dag) \after m^\dag \text{  and  }
      m \after (x^* \times 1) = (x^\dag \times 1) \after m^\dag\\
      \text{ for all } x \colon 1 \relto X.\end{array}
  \]
  \[\vcenter{\hbox{\begin{tikzpicture}[scale=0.5,thick]
        \node (0) at (0,2) {};
        \node [dot] (1) at (0,1) {};
        \node (2) at (-.75,0) {};
        \node [box] (3) at (.75,0) {$x^*$};
        \node (4) at (-.75,-1) {};
        \draw (0) to (1);
        \draw [out=180, in=90] (1) to (2) -- (4);
        \draw [out=0, in=90] (1) to (3);
    \end{tikzpicture}}}
    = 
    \vcenter{\hbox{\begin{tikzpicture}[scale=0.5,thick]
        \node (0) at (0,-2) {};
        \node [dot] (1) at (0,-1) {};
        \node (2) at (-.75,0) {};
        \node [box] (3) at (.75,0) {$x$};
        \node (4) at (-.75,1) {};
        \draw (0) to (1);
        \draw [out=180, in=270] (1) to (2) -- (4);
        \draw [out=0, in=270] (1) to (3);
    \end{tikzpicture}}}
    \qquad \qquad
    \vcenter{\hbox{\begin{tikzpicture}[scale=0.5,thick]
        \node (0) at (0,2) {};
        \node [dot] (1) at (0,1) {};
        \node (2) at (.75,0) {};
        \node [box] (3) at (-.75,0) {$x^*$};
        \node (4) at (.75,-1) {};
        \draw (0) to (1);
        \draw [out=0, in=90] (1) to (2) -- (4);
        \draw [out=180, in=90] (1) to (3);
    \end{tikzpicture}}}
    = 
    \vcenter{\hbox{\begin{tikzpicture}[scale=0.5,thick]
        \node (0) at (0,-2) {};
        \node [dot] (1) at (0,-1) {};
        \node (2) at (.75,0) {};
        \node [box] (3) at (-.75,0) {$x$};
        \node (4) at (.75,1) {};
        \draw (0) to (1);
        \draw [out=0, in=270] (1) to (2) -- (4);
        \draw [out=180, in=270] (1) to (3);
    \end{tikzpicture}}}
  \]
\end{definition}

In the presence of (U) and (A), condition (H) above is equivalent to
(F)~\cite{abramskyheunen:hstar}. But in the absence of (U), it is
stronger than (F), and means concretely that
\[
   \forall A \subseteq X \, \forall x,y \in X \left[
   \begin{array}{l}
     \exists a \in A .\, xa=y \iff \exists a^* \in A^* .\, ya^*=x, \\
     \exists a \in A .\, ax=y \iff \exists a^* \in A^* .\, a^*y=x.
   \end{array}\right.
\]
Equivalently,
\[
   \forall a \in A \subseteq X \, \forall x \in X \left[
   \begin{array}{l}
     xa \defined \implies \exists a^* \in A^* .\, xaa^*=x, \\
     ax \defined \implies \exists a^* \in A^* .\, a^*ax=x. \\
   \end{array}  \right.
\]

We will prove that relative H*-algebras are precisely
semigroupoids that are regular and locally cancellative. Recall that a
semigroupoid is a `category without identities', just like a
semigroup is a `non-unital monoid'~\cite{howie:semigroups}. More
precisely, a \emph{semigroupoid} consists of a diagram 
\[\xymatrix@C+3ex{
  G_0 
  & G_1 \ar@<-.6ex>|-{s}[l] \ar@<.6ex>|-{t}[l] 
  & G_1 \times_{G_0} G_1 \ar|-{m}[l]
}\]
(in the category $\Cat{Set}$ of sets and functions) such that $m
\after (m \times 1) = m \after (1 \times m)$.
We may also assume that $s$ and $t$ are jointly epic, \textit{i.e.}\
$G_0 = s(G_1) \cup t(G_1)$.
A \emph{semifunctor} is then a `functor without preservation of
identities', \textit{i.e.}\ a pair of functions $f_i \colon G_i \to
G'_i$ that commute with the above structure.

A \emph{pseudoinverse} of $f \in G_1$ is an element $f^* \in G_1$
satisfying ($s(f)=t(f^*)$ and $t(f)=s(f^*)$ and) $f=ff^*f$ and
$f^*=f^*ff^*$. A semigroupoid is \emph{regular} when every $f \in G_1$
has a pseudoinverse. Finally, a semigroupoid is \emph{locally
cancellative} when $fhh^*=gh^*$ implies $fh=g$, and $h^*hf=h^*g$
implies $hf=g$, for any $f,g,h \in G_1$
and any pseudoinverse $h^*$ of $h$. The following lemma shows that
the asymmetry in the latter condition is only apparent.

\begin{lemma}\label{locallycancellativesymmetric}
  A semigroupoid is locally cancellative if and only if $fh^*h=gh$
  implies $fh^*=g$ for any $f,g,h \in G_1$ and any pseudoinverse $h^*$
  of $h$.  
\end{lemma}
\begin{proof}
  If $h^*$ is a pseudoinverse of $h$, then $h$ is a pseudoinverse of
  $h^*$, too.    
\end{proof}

Examples of locally cancellative semigroupoids are semigroupoids in
which every morphism is both monic and epic. Clearly, groupoids are
examples of locally cancellative regular semigroupoids. The following
lemma gives a converse in the presence of identities. 

\begin{lemma}\label{locallycancellativeidentities}
  If a locally cancellative regular semigroupoids has identities, then
  it is a groupoid.
\end{lemma}
\begin{proof}
  Let $f \in G_1$. By regularity, there is a pseudoinverse $f^*$. In
  the definition of local cancellativity, take $h=f^*$ and
  $g=1_{t(f)}$, and $h^*=f$. Then $fhh^* = ff^*f = f = 1f = gh^*$, and
  so $ff^*=1$. By Lemma~\ref{locallycancellativesymmetric} similarly
  $f^*f=1$. Thus $f^*$ is an inverse of $f$, and therefore unique. 
\end{proof}

\subsection{From relative H*-algebras to semigroupoids}

For the rest of this subsection, fix a relative H*-algebra $m \colon X
\times X \relto X$.

\begin{definition}\label{def:hstartosgpd}
  Define $\cat{G}$ by
  \begin{align*}
    G_0 & = \{ f \in X \mid m(f,f)=f \}, \\
    G_1 & = X, \\
    s & = \{ (f,f^*f) \mid f^* \emph{ is a pseudoinverse of } f \}
    \colon G_1 \relto G_0 \\
    t & = \{ (f,ff^*) \mid f^* \emph{ is a pseudoinverse of } f \}
    \colon G_1 \relto G_0. 
  \end{align*}
\end{definition}

\begin{lemma}\label{regular}
   For each element $a$ in a relative H*-algebra there exists $a^* \in \{a\}^*$
   satisfying $a^*aa^*=a^*$ and $aa^*a=a$. 
\end{lemma}
\begin{proof}
  By (M), we have $\forall y \in X \,\exists a,x \in X .\,
  y=ax$. Applying (H) with $A=X$ gives $\forall a \in X
  \,\exists x \in X .\, xa\defined$. 
  Now let $a \in X$. If we substitute $A=\{a\}$, then (H) becomes
  \begin{align*}
    \forall x,y \in X & \big[ xa=y \iff 
    \exists a^* \in \{a\}^* .\, ya^*=x \big] \\
    \forall x,y \in X & \big[ ax=y \iff 
    \exists a^* \in \{a\}^* .\, a^*y=x \big] 
  \end{align*}
  As above, there exists $x \in X$ with $xa \defined$. So 
  by the first condition above, there is $a' \in \{a\}^*$ with $aa'
  \defined$. Hence, by the second condition, there is $a'' \in
  \{a\}^*$ with $a'' a a'=a'$. Applying the first condition again, now
  with $x=a'$ and $y=a''a$, gives $a'a=a''a$. Therefore we have
  $a^*=a^*aa^*$ for $a^*=a' \in \{a\}^*$. Finally, applying the first
  condition again, this time with $x=aa^*$ and $y=a$, we find that
  also $aa^*a=a$.
\end{proof}

\begin{lemma}\label{semigroupoid}
  The data $\cat{G}$ form a well-defined semigroupoid.
\end{lemma}
\begin{proof}
  By (A), the condition $m \after (m\times 1)=m \after (1 \times m)$ is clearly
  satisfied. It remains to prove that $m$, $s$ and $t$ are
  well-defined functions. 

  The former means that $(g,f) \in G_1 \times_{G_0} G_1$
  implies $gf \defined$. Assume $s(g)=t(f)$, \textit{i.e.}~$g^*g=ff^*$ for some
  pseudoinverses $g^*$ and $f^*$ of $g$ and $f$. Because $g^*g$ is
  idempotent, we have $g^*gff^* = g^*gg^*g = g^*g \defined$, and
  therefore also $gf \defined$. Hence $m$ is well-defined.

  As for $t$, suppose that $f^*$ and $f'$ are both pseudoinverses of
  $f$, so that $(f,ff^*) \in s$ and $(f,ff') \in s$. Then $ff^*f = f =
  ff'f$. Set $A=\{f^*\}$, $a=f^*$, $x=f$, and $y=ff'$. By
  Lemma~\ref{regular}, we obtain $f \in A^*$, and so $ya^*=x$ for $a^*=f$.
  Now it follows from (H) that $ff^*=xa=y=ff'$. 
  Similarly, $s$ is a well-defined function. 
\end{proof}

\begin{theorem}\label{locallycancellative}
  If $m$ is a relative H*-algebra, then $\cat{G}$ is a locally
  cancellative regular semigroupoid.
\end{theorem}
\begin{proof}
  Regularity is precisely Lemma~\ref{regular}. Suppose that
  $fhh^*=gh^*$ for a pseudoinverse $h^*$ of $h$. Applying (H) to
  $A=\{h\}$, $x=fhh^*$, $y=g$, $a=h$ and $a^*=h^*$ yields
  $fh=fhh^*h=xa^*=y=g$. Hence $\cat{G}$ is locally cancellative.
\end{proof}

\subsection{From semigroupoids to relative H*-algebras}

For the rest of this subsection, fix a semigroupoid $\cat{G}$.

\begin{definition}
  Define
  \begin{align*}
    X & = G_1, \\
    m & = \{ (g,f,gf) \mid s(g)=t(f) \} \colon G_1 \times G_1 \relto G_1, \\
    A^* & = \{a^* \in X \mid a^*aa^*=a^* \mbox{ and } aa^*a=a \mbox{ for all } a \in A \}.
  \end{align*}
\end{definition}

\begin{theorem}
  If $\cat{G}$ is a locally cancellative regular semigroupoid, then
  $m$ is a relative H*-algebra.  
\end{theorem}
\begin{proof}
  Clearly, (A) is satisfied. Because
  \[
    m^\dag \after m 
   = \{ (f,f) \in G_1^2 \mid \exists (g,h) \in G_2 \,.\, f=hg \}
  \]
  we have $m^\dag \after m \subseteq 1$. Conversely, if $f \in G_1$,
  setting $g=f$ and $h=f^*f$ for some pseudoinverse $f^*$ of $f$, then
  $f=gh$. Hence (M) is satisfied. 

  Finally, we verify (H). Let $A \subseteq X$ be given, let $a \in
  A$ and $x \in X$, and suppose that $xa\defined$. That means that
  $s(f)=t(a)$. By regularity, $a$ has a pseudoinverse $a^* \in A^*$,
  and we have $xa=xaa^*a$. Setting $f=xa$, $g=x$, $h=a$ and $h^*=a^*$
  in the definition of local cancellativity yields $xaa^*=x$. The
  symmetric condition is verified similarly. Hence (H) is satisfied. 
\end{proof}

\subsection{Functoriality}

This subsection proves that the assignments $m \mapsto \cat{G}$ and
$\cat{G} \mapsto m$ extend functorially to an adjunction. The
following definitions give well-defined categories, just 
as in Subsection~\ref{subsec:frob:functoriality}. Condition (I) has to
be adapted to the non-unital setting of H*-algebras, and becomes the
following.
\begin{align}
  y \after r \after x = y^* \after r \after x^*
  \text{  for all }x \colon 1 \relto X, y \colon 1 \relto Y
  \tag{I'}
\end{align}
Concretely, this means that $(x,y) \in r$ if and only if $(x^*,y^*) \in
r$ for any pseudoinverses $x^*$ of $x$ and $y^*$ of $y$.

\begin{definition}
  Relative H*-algebras can be made into the objects of several categories.
  A morphism $(X,m_X) \to (Y,m_Y)$ in $\Cat{Hstar}(\Cat{Rel})\rel$ is
  a morphism $r \colon X \relto Y$ in $\Cat{Rel}$ satisfying (R) and (I').
  A morphism in $\Cat{Hstar}(\Cat{Rel})$ is a morphism $r \colon X
  \relto Y$ in $\Cat{Rel}$ that satisfies (I') and preserves multiplication: $r \after
  m_X = m_Y \after (r \times r)$. A morphism in
  $\Cat{Hstar}(\Cat{Rel})\func$ is a morphism in
  $\Cat{Hstar}(\Cat{Rel})$ that additionally has a right adjoint.
\end{definition}

\begin{definition}
  Locally cancellative regular semigroupoids can be made into the
  objects of several categories.
  Morphisms in $\Cat{LRSgpd}$ are semifunctors.
  Morphisms in $\cat{G} \to \cat{H}$ in $\Cat{LRSgpd}\mfunc$ are
  multi-valued semifunctors, \textit{i.e.}\ multi-valued functions
  $G_i \to H_i$ satisfying only the first condition of
  Definition~\ref{def:mfunc}. 
  Morphisms in $\Cat{LRSgpd}\rel$ are
  locally cancellative regular subsemigroupoids of $\cat{G} \times
  \cat{H}$.
\end{definition}

\begin{proposition}
  The assignments $m \mapsto \cat{G}$ and $\cat{G} \mapsto m$ extend to functors
  \begin{align*}
    \Cat{Hstar}(\Cat{Rel})\rel & \leftrightarrows \Cat{LRSgpd}\rel, \\
    \Cat{Hstar}(\Cat{Rel}) & \leftrightarrows \Cat{LRSgpd}\mfunc, \\
    \Cat{Hstar}(\Cat{Rel})\func & \leftrightarrows \Cat{LRSgpd}. 
  \end{align*}
\end{proposition}
\begin{proof}
  We prove the case $\Cat{Hstar}(\Cat{Rel})\rel 
  \leftrightarrows \Cat{LRSgpd}\rel$.
  Let $(X,m_X)$ and $(Y,m_Y)$ be relative H*-algebras, inducing
  locally cancellative regular semigroupoids $\cat{G}$ and
  $\cat{H}$. Given $r \colon m_X \to m_Y$, define $m_r \colon r \times
  r \relto r$ as in the proof of Theorem~\ref{gpdfrobext}; it
  satisfies (A) and (M). It also satisfies (H), as we now verify. For
  $A \subseteq r$, take
  $
    A^* = \{ (x^*,y^*) \mid (x,y) \in A, x^* \in \{x\}^*, y^* \in
    \{y\}^* \}
  $.
  \begin{align*}
    & (1 \times A) \after m_r^\dag \\
    & \: = \{((x,y),(a,b)) \in r \times r \mid \exists (c,d) \in A \,.\, y=bd, x=ac\} \\
    & \stackrel{\mbox{\tiny{(H)}}}{=} \{ ((x,y),(a,b)) \in r \times r \mid \exists
    (c,d) \in A, c^* \in \{c\}^*, d^* \in \{d\}^* \,.\, a=xc^*, b=yd^*
    \} \\
    & \: = m_r \after (1 \times A^*).
  \end{align*}
  Theorem~\ref{locallycancellative} now shows that $m_r$ induces a
  subsemigroupoid of $\cat{G} \times \cat{H}$.
  Conversely, if $\cat{R}$ is a subsemigroupoid of $\cat{G} \times
  \cat{H}$, then $R_1 \colon G_1 \relto H_1$ clearly satisfies
  (R) and (I'). Finally, the identity relation $r \colon m_X \relto m_Y$
  corresponds to the diagonal subsemigroupoid, which is indeed regular
  and locally cancellative. 
  These constructions clearly restrict to the subcategories of the statement.
\end{proof}

\begin{theorem}\label{hstaradjunctions}
  The functors from the previous proposition form adjunctions.
  \[\xymatrix@C+3ex@R-5ex{
    \Cat{LRSgpd}\rel \ar@<1ex>[r] \ar@{}|-{\perp}[r]
    & \Cat{Hstar}(\Cat{Rel})\rel \ar@<1ex>[l] \\
    \Cat{LRSgpd}\mfunc \ar@<1ex>[r] \ar@{}|-{\perp}[r]
    & \Cat{Hstar}(\Cat{Rel}) \ar@<1ex>[l] \\
    \Cat{LRSgpd} \ar@<1ex>[r] \ar@{}|-{\perp}[r] 
    & \Cat{Hstar}(\Cat{Rel})\func \ar@<1ex>[l]
  }\]
\end{theorem}
\begin{proof}
  Starting with a relative
  H*-algebra $m \colon X \times X \relto X$, we end up with 
  \[
  \{(g,f,gf) \mid \exists g^* \in \{g\}^* \exists f^* \in \{f\}^*
  \,.\, g^*g=ff^* \} \colon X \times X
  \relto X.
  \]
  Clearly this is a subrelation of $m$, and the inclusion forms the
  unit of the adjunction.

  Starting with a locally cancellative regular semigroupoid $\cat{G}$,
  we end up with
  \[\xymatrix@C+3ex{
    \{ f \in G_1 \mid f^2=f \}
    & G_1 \ar@<-.6ex>|-{s'}[l] \ar@<.6ex>|-{t'}[l] 
    & G_1 \times_{s',t'} G_1 \ar|-{m}[l]
  }\]
  where $s'(f)=f^*f$ and $t'(f)=ff^*$. Clearly, the original $\cat{G}$
  maps into this, giving the counit of the adjunction.
  Naturality and the triangle equations are easily checked for
  $\Cat{LRSgpd}\rel \leftrightarrows
  \Cat{Hstar}(\Cat{Rel})\rel$. Because the unit and counit are
  functions, the statement also holds for the other subcategories.
\end{proof}

\begin{proposition}
  The largest subcategories making the 
  adjunctions of Theorem~\ref{hstaradjunctions} into equivalences are
  $\Cat{Gpd}\rel$ and $\Cat{Frob}(\Cat{Rel})\rel$, and their variations. In that
  case, the equivalences are in fact isomorphisms.
\end{proposition}
\begin{proof}
  Consider the counit of the proof of Theorem~\ref{hstaradjunctions}.
  It is an isomorphism precisely when $G_0$ coincides with the
  idempotents of $G_1$. But then the unique idempotent on $x \in G_0$
  is an identity, and $\cat{G}$ is a groupoid by
  Lemma~\ref{locallycancellativeidentities}. In other
  words, $\Cat{Gpd}$ and its variations are the largest subcategories
  of $\Cat{LRSgpd}$ and its variations turning the adjunctions
  into reflections.

  Next consider the unit of the adjunctions. It is an isomorphism
  when $gf \defined$ implies $g^*g=ff^*$ for some
  pseudoinverses $f^* \in \{f\}^*$ and $g^* \in \{g\}^*$. In that case we
  can define a unit for the H*-algebra $(X,m)$ by $\{u \in X \mid u=u^*u\}$, for if
  $uf \defined$ and $u=u^*u$ then $uf=u^*uf=ff^*f=f$. But recall that unital
  relative H*-algebras are relative Frobenius algebras.
  In other words, $\Cat{Frob}(\Cat{Rel})$ and its variations
  are the largest subcategories of $\Cat{Hstar}(\Cat{Rel})$ and its
  variations turning the above adjunctions into reflections.
\end{proof}

\section{Groupoids and semigroupoids}
\label{sec:quotient}

The forgetful functor $\Cat{Gpd} \to \Cat{Cat}$ has a left
adjoint, that freely adds inverses. Similarly, the forgetful functor
$\Cat{Cat} \to \Cat{Sgpd}$ to the category of semigroupoids and
semifunctors has a left adjoint, that freely adds identities. The
image of the latter left adjoint consists precisely of those
categories in which the only isomorphisms are identities. 
Hence there is a functor $\Cat{Sgpd} \to \Cat{Gpd}$
giving the free groupoid on a semigroupoid. Restricting it gives a
functor that turns a locally cancellative regular semigroupoid into a
groupoid. 
\[\xymatrix@C+3ex{
  \Cat{LRSgpd} \ar@<1ex>[r] \ar@{}|-{\perp}[r]
  & \Cat{Gpd} \ar@{_{(}->}@<1ex>[l]
}\]
The morphisms in these categories are (semi)functors. This section
establishes right adjoints to the inclusion $\Cat{Gpd}\rel 
\hookrightarrow \Cat{LRSgpd}\rel$ and its variations with other
choices of morphisms. This is then applied to obtain 
adjunctions between $\Cat{HStar}(\Cat{Rel})\rel$ and
$\Cat{Frob}(\Cat{Rel})\rel$ (and their variations). 
The idea in building the right adjoint is to identify all
idempotents: a group is a regular semigroup with a single
idempotent. 

\begin{definition}\label{def:sgpdtogpd}
  For a semigroupoid $\cat{G}$, define $\sim$ as the
  congruence (see~\cite[Section~II.8]{maclane:categories}) 
  generated by $f \sim g$ when $s(f)=s(g)$ and $f^2=f$ and
  $g^2=g$. Set
  \begin{align*}
    G'_0 & = G_0, 
    & s'([f]) & = s(f), 
    & m'([g],[f]) & = [m(g,f)], \\
    G'_1 & = G_1 \mathop{\slash} \mathop{\sim}, 
    & t'([f]) & = t(f).
  \end{align*}
\end{definition}

\begin{lemma}
  If $\cat{G}$ is a (locally cancellative regular) semigroupoid, then 
  \[\cat{G'} = \left(\xymatrix@C+3ex{
    G'_0 
    & G'_1 \ar@<-.6ex>|-{s'}[l] \ar@<.6ex>|-{t'}[l] 
    & G'_1 \times_{G'_0} G'_1 \ar|-{m'}[l]
  }\right)\]
  is again a well-defined (locally cancellative regular) semigroupoid.
\end{lemma}
\begin{proof}
  Because $\sim$ is a congruence, $m'$ is
  associative~\cite[Proposition~II.8.1]{maclane:categories}). Because
  $s$ and $t$ are jointly epic, so are $s'$ and $t'$. Hence $\cat{G'}$
  is a semigroupoid. 
  If $\cat{G}$ is regular, then $[f^*]$ is a pseudoinverse for $[f]
  \in G'_1$, where $f^*$ is any pseudoinverse of $f$ in $G_1$, and so
  $\cat{G'}$ is regular. 
  Finally, it is easy to see that $\cat{G'}$ inherits local
  cancellativity from $\cat{G}$ using that $\sim$ is a congruence.
\end{proof}

\begin{lemma}
  If the semigroup $\cat{G}$ is locally cancellative and regular, then
  \begin{align*}
    F(\cat{G})=\left(\xymatrix@C+3ex{
        G'_0 \ar|-{e'}[r] 
        & G'_1 \ar@(ul,ur)|-{i'} \ar@<-1.2ex>|-{s'}[l] \ar@<1.2ex>|-{t'}[l] 
        & G'_1 \times_{G'_0} G'_1 \ar|-{m'}[l]
      }\right)
  \end{align*}
  is a well-defined groupoid, where 
  \[
    e'(s(f)) = [f^* f], \qquad
    e'(t(f)) = [ff^*], \qquad
    i'([f]) = [f^*].
  \]
\end{lemma}
\begin{proof}
  Because $\cat{G}$ is regular, it makes sense to speak about
  $f^*$. Because $G'_0 = \mathrm{Im}(s') \cup \mathrm{Im}(t')$, the
  above prescription completely defines $e'$. Finally, $e'$ is
  well-defined, for if $s(f)=s(g)$, then $f^*f \sim g^* g$ because
  both are idempotent. Similarly, if $s(f)=t(g)$, then $f^*f \sim gg^*$.

  We now show that $i$ is a well-defined function. Suppose that $g$
  and $h$ are pseudo-inverses of $f$. Then $gf$ and $hf$ are
  idempotent. Also $s(gf)=s(hf)=t(gf)=t(hf)$, so $[gf]=[hf]$. But then
  $[g]=[h]$ by local cancellativity of $\cat{G'}$. 

  By construction, these data makes $F(\cat{G})$ into a groupoid.
\end{proof}

\begin{proposition}
  The assignment $\cat{G} \mapsto F(\cat{G})$ of the previous lemma
  extends to functors $F\rel \colon \Cat{LRSgpd}\rel \to
  \Cat{Gpd}\rel$, $F\mfunc \colon \Cat{LRSgpd}\mfunc \to
  \Cat{Gpd}\mfunc$, and $F \colon \Cat{LRSgpd} \to \Cat{Gpd}$.
\end{proposition}
\begin{proof}
  Let $\cat{R}$ be a morphism $\cat{G} \to \cat{H}$ in
  $\Cat{LRSgpd}\rel$. Then it is subsemigroupoid of $\cat{G} \times
  \cat{H}$, and hence a semigroupoid in its own right. Hence we can
  define $F\rel(\cat{R})$ as in the previous lemma. It is easy to see that
  $F\rel(\cat{R})$ is a subsemigroupoid of $F\rel(\cat{G}) \times F\rel(\cat{H})$.
  Finally, it clear that $F\rel$ preserves identities and composition,
  and restricts to give functors $F\mfunc$ and $F$. 
\end{proof}

\begin{theorem}
  The inclusion $\Cat{Gpd} \hookrightarrow \Cat{LRSgpd}$ has $F$ as a
  right adjoint. \\
  The inclusion $\Cat{Gpd}\mfunc \hookrightarrow \Cat{LRSgpd}\mfunc$
  has $F\mfunc$ as a right adjoint. \\
  The inclusion $\Cat{Gpd}\rel \hookrightarrow \Cat{LRSgpd}\rel$ has
  $F\rel$ as a right adjoint..  
\end{theorem}
\begin{proof}
  We start by exhibiting the unit of the adjunctions. Let $\cat{G}$ be
  a locally cancellative regular semigroupoid. Then $G_0 =
  F(\cat{G})_0$, and there is a projection function $G_1
  \twoheadrightarrow (G_1 \mathop{\slash} \mathop{\sim}) =
  F(\cat{G})_1$. By construction of $s'$, $t'$ and $m'$, this induces
  a semifunctor $\cat{G} \to F(\cat{G})$, and hence a subsemigroupoid
  of $\cat{G} \times F(\cat{G})$. Because $\cat{G}$ is locally
  cancellative and regular itself, this subsemigroupoid is a
  well-defined morphism $\cat{G} \to F(\cat{G})$ in
  $\Cat{LRSgpd}\rel$ as well as in $\Cat{LRSgpd}\mfunc$ and
  $\Cat{LRSgpd}$. It is easy to see that this is natural in $\cat{G}$. 

  As for the counit, notice that if $\cat{G}$ is a groupoid, then $G_1
  \cong (G_1 \mathop{\slash} \mathop{\sim})$. So the subsemigroupoid
  of $\cat{G} \times F(\cat{G})$ is in fact a groupoid, and hence
  gives a well-defined morphism $F(\cat{G}) \to \cat{G}$ in
  $\Cat{Gpd}\rel$, $\Cat{Gpd}\mfunc$ and $\Cat{Gpd}$, that is natural
  in $\cat{G}$.  

  One readily verifies that this unit and counit satisfy the triangle
  equations.
\end{proof}

Thus the functor $F$ provides a universal way to pass from locally
regular semigroupoids to groupoids. Restriction to the one-object case
shows that collapsing all idempotents turns a locally cancellative regular
semigroup into a group.

\begin{corollary}
  There are adjunctions
  \[\xymatrix@C+3ex@R-2ex{
    \Cat{Hstar}(\Cat{Rel})\func \ar@{^{(}->}[r] \ar@<1ex>[d]
    & \Cat{Hstar}(\Cat{Rel}) \ar@{^{(}->}[r] \ar@<1ex>[d]
    & \Cat{Hstar}(\Cat{Rel})\rel \ar@<1ex>[d] \\
    \Cat{Frob}(\Cat{Rel})\func \ar@{^{(}->}[r] \ar@<1ex>[u] \ar@{}|-{\dashv}[u]
    & \Cat{Frob}(\Cat{Rel}) \ar@{^{(}->}[r] \ar@<1ex>[u] \ar@{}|-{\dashv}[u]
    & \Cat{Frob}(\Cat{Rel})\rel \ar@<1ex>[u] \ar@{}|-{\dashv}[u] 
  }\]
  Explicitly, the right adjoints map a relative H*-algebra $(X,m_X)$
  to $(X \mathop{\slash} \mathop{\sim}, m')$, where $\sim$ is the
  equivalence relation generated by $f \sim g$ if $f^2=f$ and $g^2=g$
  and $gf \defined$, and $m'([g],[f])=[m(g,f)]$.
\end{corollary}
\begin{proof}
  Simply compose the following adjunctions, and similarly for the
  other choices of morphisms. 
  \[\xymatrix@C+2ex{
    \Cat{Frob}(\Cat{Rel})\rel \ar@<1ex>[r] \ar@{}|-{\cong}[r]
    & \Cat{Gpd}\rel \ar@<1ex>[l] \ar@{_{(}->}@<1ex>[r] \ar@{}|-{\perp}[r]
    & \Cat{LRSgpd}\rel \ar@<1ex>^-{F\rel}[l] \ar@<1ex>[r] \ar@{}|-{\perp}[r]
    & \Cat{Hstar}(\Cat{Rel})\rel \ar@<1ex>[l]
  }\]
  Applying Definitions~\ref{def:hstartosgpd}, \ref{def:sgpdtogpd},
  and~\ref{def:gpdtofrob} in order to a given relative H*-algebra results
  in the relative Frobenius algebra of the statement.
\end{proof}

\bibliographystyle{plain}
\bibliography{relfrobgpd}

\end{document}